\documentclass[11pt]{amsart}

\usepackage[a4paper,hmargin=3.5cm,vmargin=3.5cm]{geometry}
\usepackage{amsfonts,amssymb,amscd,amstext}
\usepackage{graphicx}
\usepackage[dvips]{epsfig}
\usepackage{quoting}
\usepackage{hyperref}

\usepackage{fancyhdr}
\input xy
\xyoption{all}

\usepackage{enumerate}
\usepackage{titlesec}
\usepackage{mathrsfs}

\titleformat{\section}
{\filcenter\bfseries\large} {\thesection{.}}{0.2cm}{}
\titleformat{\subsection}[runin]
{\bfseries} {\hspace*{4mm}\thesubsection{.}}{0.15cm}{}[.]
\titleformat{\subsubsection}[runin]
{\em}{\thesubsubsection{.}}{0.15cm}{}[.]

\usepackage[up,bf]{caption}

\newtheorem{theorem}{Theorem}[section]

\newtheorem{lemma}[theorem]{Lemma}

\theoremstyle{definition}
\newtheorem{definition}[theorem]{Definition}
\newtheorem{remark}[theorem]{Remark}

\numberwithin{equation}{section}
\numberwithin{figure}{section}

\newcommand\Acal{\mathcal{A}}

\newcommand\Ccal{\mathcal{C}}

\newcommand\Hcal{\mathcal{H}}

\newcommand\Kcal{\mathcal{K}}

\newcommand\Mcal{\mathcal{M}}

\newcommand\Rcal{\mathcal{R}}

\newcommand\C{\mathbb{C}}
\newcommand\D{\overline{\mathbb D}}

\renewcommand\D{\mathbb D}

\newcommand\N{\mathbb{N}}

\newcommand\Z{\mathbb{Z}}

\renewcommand\b{\mathbb{B}}
\renewcommand\c{\mathbb{C}}
\newcommand\cd{\overline{\mathbb D}}
\newcommand\cp{\mathbb{CP}}
\renewcommand\d{\mathbb D}

\newcommand\n{\mathbb{N}}
\renewcommand\r{\mathbb{R}}

\renewcommand\t{\mathbb{T}}
\newcommand\z{\mathbb{Z}}

\newcommand\mgot{\mathfrak{m}}

\newcommand\wh{\widehat}

\newcommand\ol{\overline}

\newcommand\dist{\mathrm{dist}}
\renewcommand\span{\mathrm{span}}
\newcommand\length{\mathrm{length}}

\newcommand\IE{\mathcal{I}_{E}}
\newcommand\Bo{\mathcal{B}}

\def\dist{\mathrm{dist}}
\def\span{\mathrm{span}}
\def\length{\mathrm{length}}

\usepackage{color}

\begin{document}


\fancyhead[LO]{Complete meromorphic curves with Jordan boundaries}
\fancyhead[RE]{T.\ Vrhovnik}
\fancyhead[RO,LE]{\thepage}

\thispagestyle{empty}


\begin{center}

{\bf\LARGE Complete meromorphic curves \protect\\ with Jordan boundaries}

\medskip

%
%
{\large\bf Tja\v{s}a Vrhovnik}
\end{center}


%
%
\medskip

\begin{quoting}[leftmargin={5mm}]
{\small
\noindent {\bf Abstract}\hspace*{0.1cm}
We prove that given a finite set $E$ in a bordered Riemann surface $\Rcal$, there is a continuous map $h\colon \overline\Rcal\setminus E\to\C^n$ ($n\geq 2$) such that $h|_{\Rcal\setminus E} \colon \Rcal\setminus E\to\C^n$ is a complete holomorphic immersion (embedding if $n\geq 3$) which is meromorphic on $\Rcal$ and has effective poles at all points in $E$, and $h|_{b\overline\Rcal} \colon b\overline\Rcal\to\C^n$ is a topological embedding. In particular, $h(b\overline\Rcal)$ consists of the union of finitely many pairwise disjoint Jordan curves which we ensure to be of Hausdorff dimension one. We establish a more general result including uniform approximation and interpolation.


\noindent{\bf Keywords}\hspace*{0.1cm} 
complete complex curve,
meromorphic curve,
Riemann surface


\noindent{\bf Mathematics Subject Classification (2020)}\hspace*{0.1cm} 
32H02, 
32H04, 
32E30, 
32B15. 
}
\end{quoting}



\section{Introduction}\label{sec:intro}

\noindent 
A \emph{bordered Riemann surface} $\Rcal$ is an open connected Riemann surface which is the interior of a compact one dimensional smoothly bounded complex manifold $\ol{\Rcal}$, whose boundary $b\ol{\Rcal} = \ol{\Rcal}\setminus \Rcal$ consists of finitely many Jordan curves. Its closure $\ol{\Rcal}$ is called a \emph{compact bordered Riemann surface}. We denote by $\D$ the open unit disc in $\C$.

\emph{Yang problem} from 1977 asks whether there exist complete immersed complex submanifolds $\phi \colon M^k \to \c^n$ $(1\leq k < n)$ with bounded image \cite{Yang1977,yang1977curvatures}. Recall that $\phi$ is \emph{complete} if the Riemannian metric on $M$ induced by the Euclidean metric on $\C^n$ via $\phi$ is complete in the classical sense. The first positive answer to Yang's question was given by P. W. Jones in 1979 \cite{jones1979complete}. He constructed a complete bounded holomorphic immersion $\D \to \C^2$, a complete bounded holomorphic embedding $\D \to \C^3$, and a complete proper holomorphic embedding of the disc to the unit ball of $\C^4$.
We refer the reader to \cite{alarcon2022yang}, where a brief history of the problem and most important results are presented.

Jones' existence theorems were extended only in 2013 by A. Alarc{\'o}n and F. J. L{\'o}pez \cite{alarcon2013null}, who constructed examples of complete bounded immersed complex curves in $\C^2$ and embedded in $\C^3$ with arbitrary topology; see also \cite{alarcon2014null}. However, their method does not enable to control the complex structure of the curve, except of course in the simply-connected case, when every such curve must be biholomorphic to the disc. This more difficult task was carried out by A. Alarc{\'o}n and F. Forstneri{\v{c}}, who, using more powerful complex analytic tools, in particular, the Riemann-Hilbert method and the technique for exposing boundary points \cite{forstneric2009bordered}, proved that every bordered Riemann surface admits a complete proper holomorphic immersion to the unit ball of $\c^2$ and a complete proper holomorphic embedding to the unit ball of $\c^3$ \cite[Theorem 1]{alarcon2013every}.
Using a similar approach, but with more precision, A. Alarc{\'o}n, B. Drinovec Drnov{\v{s}}ek, F. Forstneri{\v{c}} and F. J. L{\'o}pez constructed complete immersed complex curves in $\C^2$ and embedded in $\C^3$, normalized by any given bordered Riemann surface and being bordered by Jordan curves \cite[Theorem 1.6]{alarcon2015every}. 

Our first main result, which we prove in Section \ref{sec:proof-main}, generalizes the aforementioned statement from \cite{alarcon2015every} by enabling the existence of poles, and providing interpolation of a given initial map at the polar set and at a given finite set. 
Here is its precise formulation.
\begin{theorem}\label{th:main}
Let $\Rcal$ be a bordered Riemann surface and $E, \Lambda \subset \Rcal$ be finite disjoint subsets. Let $h \colon \ol{\Rcal}\setminus E \to \c^n$ ($n \geq 2$) be a continuous map which is meromorphic on $\Rcal$ having poles at all points in $E$.
Then for any number $\varepsilon > 0$ there exists a continuous map $\hat{h} \colon \ol{\Rcal} \setminus E \to \c^n$, which is a meromorphic immersion on $\Rcal$, satisfying the following properties:
\begin{itemize}
\item[\rm (i)] the difference $h-\hat{h}$ is holomorphic on $\Rcal$, vanishing to any given finite order on $E \cup \Lambda$;

\smallskip
\item[\rm (ii)] $|h-\hat{h}|(p) < \varepsilon \ \textrm{for all }p \in \ol{\Rcal}$;

\smallskip
\item[\rm (iii)] $\hat{h} \colon \Rcal \setminus E \to \c^n$ is complete;

\smallskip
\item[\rm (iv)] $\hat{h}|_{b\ol{\Rcal}} \colon b\ol{\Rcal} \to \c^n$ is a topological embedding and $\hat{h}(b\ol{\Rcal})$ consists of a pairwise disjoint union of finitely many Jordan curves of Hausdorff dimension one.
\end{itemize}
Furthermore, if $n\geq 3$, then $\hat{h} \colon \ol{\Rcal} \setminus E \to \c^n$ can be chosen an embedding with these properties, provided that $h|_{\Lambda}$ is injective.
\end{theorem}

Note that, by condition (i), $\hat{h}$ has effective poles at all points in $E$, and only there; in particular, $\lim_{p\to E}|h(p)|=\lim_{p\to E}|\hat h(p)|=+\infty$. We emphasize that $\lim_{p\to E}|\hat h(p)-h(p)|=0$ in view of (i).
Coming back to the Yang problem, observe that the complete meromorphic curves given by Theorem \ref{th:main} are bounded outside any neighbourhood of the polar set $E$.
Also note that choosing the sets $E$ and $\Lambda$ to be empty we obtain \cite[Theorem 1.6]{alarcon2015every}.

Our second main result is Theorem \ref{th:countably}, stated and proved in Section \ref{sec:countably}. Its formulation is almost analogous to the above theorem, the main difference being the source domain -- a complement in a compact Riemann surface of a countable union of pairwise disjoint smoothly bounded closed discs. However, due to some technical limitations, the initial map has to be defined on a slightly bigger domain than the approximating map provided by the theorem. Likewise, this generalizes a result by A. Alarc{\'o}n and F. Forstneri{\v{c}} \cite[Theorem 1.8]{alarcon2021calabi} by introducing poles into the picture, and granting interpolation at the poles and at a given finite set. Our method of proof broadly follows that in \cite{alarcon2015every,alarcon2021calabi}, but proceeding with extra precision and introducing an additional idea which enables us to deal with poles.


\section{Technical lemmas}

\subsection{Preliminaries and notation}\label{subsec:pre}

\noindent
We denote by $|\cdot|$ the Euclidean norm and by $\length(\cdot)$ the Euclidean length in $\c^n$. 
For a vector $v\in \c^n$, let $\span \{v\} = \{tv \colon t\in \c \}$ be its linear span.
Further, let $\d = \{t\in \c \colon |t|<1\}$ denote the open unit disc and $\t = b\ol{\d}=\{t\in \c \colon |t|=1\}$ the unit circle in $\c$.
If $f \colon K \to \c^n$ is a continuous map of a compact topological space $K$, then we set $||f||_{0,K}=\sup_{p\in K}|f(p)|$ for the maximum norm of $f$ on $K$.

Let $\Rcal$ be a bordered Riemann surface. A domain $\Mcal$ in $\Rcal$, which is itself a bordered Riemann surface, is called a \emph{bordered domain}, hence $\Mcal$ has smooth boundary. Any bordered Riemann surface $\Rcal$ can be realized as a relatively compact bordered domain in a larger open Riemann surface $\wh{\Rcal}$. Throughout this paper, we shall always assume this situation; therefore it makes sense to talk about holomorphic or meromorphic maps on $\ol{\Rcal}$.
We denote by $\Bo(\Rcal)$ the family of bordered domains $\Mcal \Subset \Rcal$ such that the closure $\ol{\Mcal}$ is Runge in $\Rcal$ and the set $\Rcal \setminus \ol{\Mcal}$ is a finite union of pairwise disjoint open annuli. In particular, $\ol{\Mcal}$ is a strong deformation retract of $\Rcal$. Recall that a compact subset of an open Riemann surface is \emph{Runge} (or \emph{holomorphically convex}) if its complement has no relatively compact connected components.

Given a compact set $X \subset \ol{\Rcal}$, we denote by $\Acal(X)^n = \Acal(X,\c^n)$ the family of continuous maps $f\colon X\to \c^n$ that are holomorphic in the interior $\mathring{X}$.

If $f \colon M \to \c^n$ is an immersed submanifold, then 
\begin{equation*}
\dist_{f}(p,q)=\inf \{ \length(f\circ \gamma) \colon \gamma \subset M \ \textrm{path connecting } p \ \textrm{and } q\}
\end{equation*}
denotes the intrinsic Riemannian distance in $M$ induced by the Euclidean distance in $\c^n$ via $f$.
An immersion $f \colon M \to \c^n$ is said to be \emph{complete} if the path $f \circ \gamma$ has infinite Euclidean length in $\c^n$ for any divergent path $\gamma$ in $M$. 
Recall that a path $\gamma \colon [0,1) \to M$ is \emph{divergent} if $\gamma(t) \notin K$ for any compact set $K\subset M$ as $t\to 1$.
We denote by $\Ccal_{d}(M,p_0)$ the family of divergent paths in $M$ with the initial point $p_0 \in M$. 
If $M$ is the interior $M=\ol{M} \setminus bM$ of a compact manifold $\ol{M}$ with nonempty boundary $bM$, then we also have the bigger family $\Ccal_{qd}(M,p_0)$ of quasidivergent paths in $M$ with the initial point $p_0 \in M$. 
A path $\gamma \colon [0,1) \to M$ is called \emph{quasidivergent}, if there exist a point $p \in bM$ and an increasing sequence of numbers $0<t_1<t_2<\dots<1$ such that $\lim_{i\to \infty}t_i=1$ and $\lim_{i\to \infty}\gamma(t_i)=p\in bM$.
Note that by compactness of $\ol{M}$, a path $[0,1) \to M$ is quasidivergent if and only if its image is not relatively compact in $M$.
Moreover, we denote the intrinsic diameter of $(\ol{M},f)$ with respect to a base point $p_0 \in M$ as
\begin{equation*}
\dist_f(p_0,bM) = \inf \{\length(f\circ \gamma) \colon \gamma \subset M \ \textrm{divergent path with } \gamma(0)=p_0\} \in [0,+\infty].
\end{equation*}

\begin{definition} \label{def:IE}
Given a finite subset $E \subset \Rcal$ of a bordered Riemann surface $\Rcal$, we shall denote by $\IE(\Rcal,\c^n)$ the set of all holomorphic maps $\ol{\Rcal} \setminus E \to \c^n$ which extend to meromorphic immersions on $\Rcal$ with poles at all points in $E$, i.e., meromorphic maps $\ol{\Rcal} \to \c^n$ with the polar set $E$ such that the restrictions to $\Rcal \setminus E$ are holomorphic immersions.
\end{definition}


\subsection{The main lemma}\label{subsec:lemma}

\noindent
The main result of this subsection is Lemma \ref{lem:main}, which is the key tool in the proofs of Theorems \ref{th:main} and \ref{th:countably}.

%
%
\begin{lemma}\label{lem:main}
Let $\Rcal$ be a bordered Riemann surface and $E, \Lambda \subset \Rcal$ be finite disjoint subsets. For any map $f \in \IE(\Rcal,\c^2)$, point $p_0 \in \Rcal \setminus (E\cup \Lambda)$, integer $d \in \Z_+$ and positive numbers $\varepsilon>0$ (small), $\tau>0$ (big) there exists a map $\hat{f}\in \IE(\Rcal,\c^2)$ satisfying:
\begin{itemize}
\item[\rm (i)] the map $\hat{f}-f$ is holomorphic on $\Rcal$, vanishing to order $d$ at every point in $E \cup \Lambda$;

\smallskip
\item[\rm (ii)] $\dist_{\hat{f}}(p_0, b\ol{\Rcal}) > \tau$;

\smallskip
\item[\rm (iii)] $|\hat{f}-f|(p) < \varepsilon \ \textrm{for all }p \in \ol{\Rcal}$;

\smallskip
\item[\rm (iv)] $\hat{f}|_{b\ol{\Rcal}} \colon b\ol{\Rcal} \to \c^2$ is a topological embedding.
\end{itemize}
\end{lemma}

Although Lemma \ref{lem:main} is stated for maps into $\c^2$, it holds true for maps into $\c^n$ for arbitrary integer $n\geq 2$. The proof is essentially the same, however, for simplicity of exposition, we will give the details only for $n=2$.
In order to prove the above result, we shall need the following technical lemma. 
The main feature of Lemma \ref{lem:dist} is enlarging the intrinsic diameter by an arbitrary positive number, while also approximating and interpolating the initial map. The proof follows the same line of ideas as those in the proof of \cite[Lemma 4.2]{alarcon2015every}, that is, pushing the images of boundary points a certain distance in a direction orthogonal to the position vector, solving a suitable Riemann-Hilbert problem and the gluing method, namely, a solution of the Cousin-I problem. However, since the given map has poles, this requires a more careful analysis. Secondly, our result provides interpolation of the initial map at the polar set $E$ and a given finite set in $\Rcal \setminus E$.

%
%
\begin{lemma}\label{lem:dist}
Let $\Rcal, \ f, \ E, \ \Lambda, \ p_0 \ \textrm{and } d$ be as in Lemma \ref{lem:main}.
Assume that $g \colon b\ol{\Rcal} \to \c^2$ is a smooth map and choose numbers $\delta>0, \ \mu>0$ such that
\begin{equation}\label{eq:f-g}
|f(p)-g(p)|<\delta \quad \textrm{for all } p\in b\ol{\Rcal} 
\end{equation}
and
\begin{equation}\label{eq:dist-f-bR}
0<\mu< \dist_{f}(p_0, b\ol{\Rcal}).
\end{equation}
Given a compact set $\Mcal \subset \Rcal$ with $E \cup \Lambda \subset \Mcal$ and positive numbers $\eta, \varepsilon >0$, there exists a map $\hat{f} \in \IE(\Rcal,\c^2)$ satisfying the following:
\begin{itemize}
\item[\rm (i)] $|\hat{f}(p)-g(p)| < \sqrt{\delta^2+\eta^2} \ \textrm{for all }p \in b\ol{\Rcal}$;

\smallskip
\item[\rm (ii)] $\dist_{\hat{f}}(p_0, b\ol{\Rcal}) > \mu + \eta$;

\smallskip
\item[\rm (iii)] the map $\hat{f}-f$ is holomorphic on $\Rcal$, vanishing to order $d$ at every point in $E \cup \Lambda$;

\smallskip
\item[\rm (iv)] $|\hat{f}-f|(p) < \varepsilon \ \textrm{for all } p\in \Mcal$.
\end{itemize}
\end{lemma}
%
%
%

%
%
\begin{proof}
Firstly, replace $\Mcal$ by a larger bordered domain in $\Bo(\Rcal)$ such that $\{p_0\} \cup E \cup \Lambda \subset \Mcal$ and without loss of generality, by \eqref{eq:dist-f-bR}, assume that
\begin{equation}\label{eq:dist-f-bM}
\dist_{f}(p_0, b\ol{\Mcal}) > \mu.
\end{equation}
Denote the points in $E$ by $e_i$, $i=1,\dots,\mgot$, so $E=\{e_i\}_{i=1}^\mgot$, and let $d_i$ be the degree of $f$ at its pole $e_i$. 
Further, assume that $\Rcal$ is a bordered domain in a bordered Riemann surface $\wh{\Rcal}$ such that $\Rcal \in \Bo(\wh{\Rcal})$.
By Mergelyan's theorem (see \cite[Theorems 1.12.7 and 1.12.14]{alarcon2021minimal}, \cite{fornaess2020holomorphic}) and general position, we may assume that $f \in \IE(\Rcal,\c^2)$ extends to $f \in \IE(\wh{\Rcal},\c^2)$.
If necessary, by general position, we may slightly modify the smooth map $g \colon b\ol{\Rcal} \to \c^2$ so that 
\begin{equation}\label{eq:f-not-g}
f(p) \neq g(p) \quad \textrm{for all } p \in b\ol{\Rcal}.
\end{equation}
For simplicity, let us assume that $b\ol{\Rcal}$ is connected; otherwise we apply the same procedure on all its connected components.
We shall construct $\hat{f}$ in four steps.

STEP 1 -- Splitting $b\ol{\Rcal}$.
Fix a positive number $\varepsilon_0$ which we shall specify later. There exist an integer $l \geq 3$ and compact connected arcs $\alpha_j \subset b\ol{\Rcal}, \ j \in \Z_l=\{0,1,\dots , l-1\}$, such that
\begin{itemize}
\item $\bigcup_{j \in \Z_l} \alpha_j = b\ol{\Rcal}$,

\smallskip
\item the arcs $\alpha_j$ and $\alpha_{j+1}$ have one common endpoint $p_j$, but are otherwise disjoint, $j\in \Z_l$,

\smallskip
\item $\alpha_j \cap \alpha_i=\emptyset \ \textrm{for } i \notin \{j-1,j,j+1\}$,

\smallskip
\item $|g(p)-g(q)|<\varepsilon_0 \ \textrm{for all } p,q \in \alpha_j$,

\smallskip
\item $|f(p)-g(q)|<\delta \ \textrm{for all } p,q \in \alpha_{j}$, and

\smallskip
\item $|f(p)-f(q)|<\varepsilon_0 \ \textrm{for all } p,q \in \alpha_j$.
\end{itemize}
The fifth property is implied by \eqref{eq:f-g}, whereas the rest are guaranteed by continuity of $f$ and $g$, provided that the arcs $\alpha_j$ are chosen sufficiently small.

For each $j\in \Z_l$ denote by $\pi_j$ the orthogonal projection
\begin{equation*}
\pi_j \colon \c^2 \to \span\{f(p_j)-g(p_j)\} \subset \c^2.
\end{equation*}
Note that the target is a complex line by condition \eqref{eq:f-not-g}.

STEP 2 -- Enlarging the distance from $p_0$ to $p_j$.
For each $j\in \Z_l$ choose an embedded Jordan arc $\gamma_j \subset \wh{\Rcal}$ such that $\gamma_j$ is transverse to $b\ol{\Rcal}$, $\gamma_j \cap \ol{\Rcal} = \{p_j\}$ and the family of arcs $\{\gamma_j\}_j$ is pairwise disjoint. By $q_j$ denote the other endpoint of $\gamma_j$.
Define 
\begin{equation*}
S=\ol{\Rcal} \cup (\cup_{j} \gamma_j) \subset \wh{\Rcal}, 
\end{equation*}
which is an admissible subset of $\wh{\Rcal}$ in the sense of \cite[Definition 1.12.9]{alarcon2021minimal}.
Let us smoothly extend $f$ to $S$ such that it satisfies
\begin{itemize}
\item[\rm (a1)] $|f(p)-g(q)|<\delta, \ |f(p)-f(q)|<\varepsilon_0 \ \textrm{for all pairs of points } (p,q)\in (\gamma_{j-1} \cup \alpha_j \cup \gamma_j)\times \alpha_j, \ j\in \Z_l$, and

\smallskip
\item[\rm (a2)] if $j\in \Z_l$ and $\{J_a\}_{a\in \Z_l}$ is a partition of $\gamma_j$ by Borel measurable sets, then $\sum_{a\in \Z_l} \length((\pi_{a}\circ f)(J_a))>2\eta$.
\end{itemize}
Observe that such an extension exists. More precisely, recalling properties of $\alpha_j$ and defining $f$ on $\gamma_j$ so that $f(\gamma_j)$ has small extrinsic diameter, it fulfills (a1). To ensure (a2), we also ask the image of $\gamma_j$ by $f$ to highly oscillate in the direction of $f(p_j)-g(p_j)$ for all $j$, taking into account the definition of $\pi_j$.

By Mergelyan's theorem on admissible sets \cite[Theorem 1.12.11 and Corollary 1.12.2]{alarcon2021minimal}, there exists a holomorphic map $\tilde{f} \colon \wh{\Rcal}\setminus E \to \c^2$, which extends to a meromorphic map on $\wh{\Rcal}$, smoothly approximates $f$ on $S$ and such that the map $\tilde{f}-f$ is holomorphic on $\wh{\Rcal}$, vanishing to order $d$ at every point in $E\cup \Lambda$. 
Further, we can assume that $\tilde{f} \colon \wh{\Rcal}\setminus E \to \c^2$ is a holomorphic immersion meeting the aforementioned conditions. 
Indeed, $\tilde{f}$ has poles at all points of $E$, so it is an immersion on a punctured open neighbourhood of each of these points. 
On the other side, a standard general position argument guarantees that $\tilde{f}$ is immersive on the complement in $\wh{\Rcal}$ of a union of open neighbourhoods of poles. 
These observations together imply that, up to a small deformation, $\tilde{f}$ is a holomorphic immersion on $\wh{\Rcal} \setminus E$.
Without loss of generality, let us replace $\tilde{f}$ by $f$.

Next, do the following for all $j\in \Z_l$. Choose small open neighbourhoods $V \subset \wh{\Rcal}$ of $S$, $W_{j}' \Subset W_j \Subset V\setminus \Mcal$ of $p_j$ and $V_j \Subset V\setminus \Mcal$ of $\gamma_j$ satisfying:
\begin{itemize}
\item[\rm (a3)] $V_j \cap \ol{\Rcal} \Subset W_{j}' \Subset W_j \Subset V\setminus \Mcal$;

\smallskip
\item[\rm (a4)] $|f(p)-g(q)|<\delta \ \textrm{for all } (p,q)\in (W_{j-1} \cup V_{j-1} \cup \alpha_j \cup V_j \cup W_j) \times \alpha_j$;

\smallskip
\item[\rm (a5)] if $\gamma_{j}' \subset W_j \cup V_j$ is an arc with one endpoint in $W_j$, the other one being $q_j$, and $\{J_a\}_{a\in \Z_l}$ is a partition of $\gamma_{j}'$ by Borel measurable sets, then $\sum_{a\in \Z_l} \length((\pi_a\circ f)(J_a))>2\eta$;

\smallskip
\item[\rm (a6)] $\ol{W}_j \cup \ol{V}_j$ are pairwise disjoint compact sets.
\end{itemize}
Properties (a4) and (a5) correspond to (a1) and (a2), respectively.
By \cite[Theorem 9.9.1]{Forstneric2017}, for every $N \in \N$ there exists a smooth diffeomorphism $\Phi \colon \ol{\Rcal} \to \Phi(\ol{\Rcal})$ onto a smoothly bounded domain $\Phi(\ol{\Rcal}) \subset V$ such that
\begin{itemize}
\item[\rm (a7)] $\Phi \colon \Rcal \to \Phi(\Rcal)$ is biholomorphic,

\smallskip
\item[\rm (a8)] $\Phi$ is as close as desired to the identity map in $\Ccal^{\infty}(\ol{\Rcal}\setminus \bigcup_{j\in \Z_l} W_{j}')$,

\smallskip
\item[\rm (a9)] $\Phi$ is tangent to the identity map to order $N$ at each point in $E \cup \Lambda$, and

\smallskip
\item[\rm (a10)] $\Phi(p_j)=q_j$ and $\Phi(\ol{\Rcal} \cap W_{j}') \subset W_j \cup V_j \ \textrm{for all } j\in \Z_l$.
\end{itemize}
Additionally, we may assume by (a8) and up to replacing $\gamma_j$ by a nearby arc that $\gamma_j \setminus \{q_j\} \subset \Phi(\Rcal \setminus \ol{\Mcal})$ holds for all $j\in \Z_l$.
Define a map
\begin{equation}\label{eq:f_0}
f_0 = f \circ \Phi \colon \ol{\Rcal} \to \c^2.
\end{equation}
If we choose $N = \max\{ \max \{d_i \colon i=1,\dots,\mgot\}, d\}$, then the map $f_0 \colon \ol{\Rcal}\setminus E \to \c^2$ is holomorphic, extending to a meromorphic immersion on $\Rcal$, and the map $f_0-f$ is holomorphic on $\Rcal$, vanishing to order $d$ at the points in $E\cup \Lambda$. Fix some $\varepsilon_1>0$ to be specified later. It turns out that
\begin{itemize}
\item $||f_0-f||_{0,\ol{\Mcal}} < \varepsilon_1$ (by (a8)),

\smallskip
\item $|f_0(p)-g(q)|<\delta, \ |f_0(p)-f(q)|<\varepsilon_0 \ \textrm{for all } p,q\in \alpha_j$ (by (a4), (a8), (a10)), and

\smallskip
\item there exists an open neighbourhood $U_j \subset \ol{\Rcal}$ of $p_j$, $U_j \cap \ol{\Mcal} = \emptyset$, such that if $\gamma \subset \ol{\Rcal}$ is an arc connecting a point in $\ol{\Mcal}$ with a point in $\ol{U}_j$ and $\{J_a\}_{a\in \Z_l}$ is a partition of $\gamma$ by Borel measurable sets, then 
	\begin{equation}\label{eq:length-f0-Ja}
	\sum_{a\in \Z_l} \length \left((\pi_a\circ f_0)(J_a)\right) > \eta.
	\end{equation}
(Recall properties (a3), (a5), (a8) and (a10).)
We may assume that $\ol{U}_j$ are simply connected, smoothly bounded and pairwise disjoint sets.
\end{itemize}

STEP 3 -- Stretching from the arcs $\alpha_j$.
Pick a number $\varepsilon_2>0$ which we shall specify later and fix a smooth retraction $r \colon \ol{\Rcal} \setminus \ol{\Mcal} \to b\ol{\Rcal}$.
For each $j\in \Z_l$ choose smoothly bounded pairwise disjoint closed discs $\ol{D}_j \subset \ol{\Rcal} \setminus \ol{\Mcal}$ such that the following conditions hold.
\begin{itemize}
\item[\rm (b1)] $|f_0(p)-g(q)|<\delta \ \textrm{for all } (p,q)\in \ol{D}_j \times \alpha_j$.

\smallskip
\item[\rm (b2)] $\ol{D}_j \cap b\ol{\Rcal}$ is a compact connected Jordan arc in $\alpha_j \setminus \{p_{j-1},p_j\}$ with endpoints in $U_{j-1}$ and $U_j$, respectively.

\smallskip
\item[\rm (b3)] $r(\ol{D}_j) \subset \alpha_j \setminus \{p_{j-1},p_j\}$ and $|(f_0 \circ r)(p)-f_0(p)|<\varepsilon_2 \ \textrm{for all } p\in \ol{D}_{j}$.
\end{itemize}
Observe that such discs exist by continuity of $f_0$ and $r$.
Let $\beta_j \Subset I_j \Subset \ol{D}_j \cap \alpha_j$ be compact connected Jordan arcs with endpoints in $U_{j-1}$ and $U_j$, respectively. Let $C_j \subset \ol{D}_j$ be an open neighbourhood of $I_j$ ($\ol{C}_j \subset \ol{D}_j$ is a disc) such that $b\ol{C}_j \cap b\ol{D}_j \subset \alpha_j$ and $b\ol{C}_j \cap \alpha_j$ lies in the relative interior of $b\ol{D}_j \cap \alpha_j$ (see Figure \ref{fig:sets}).

\begin{figure}[h!]
\begin{center}
\includegraphics[scale=0.3]{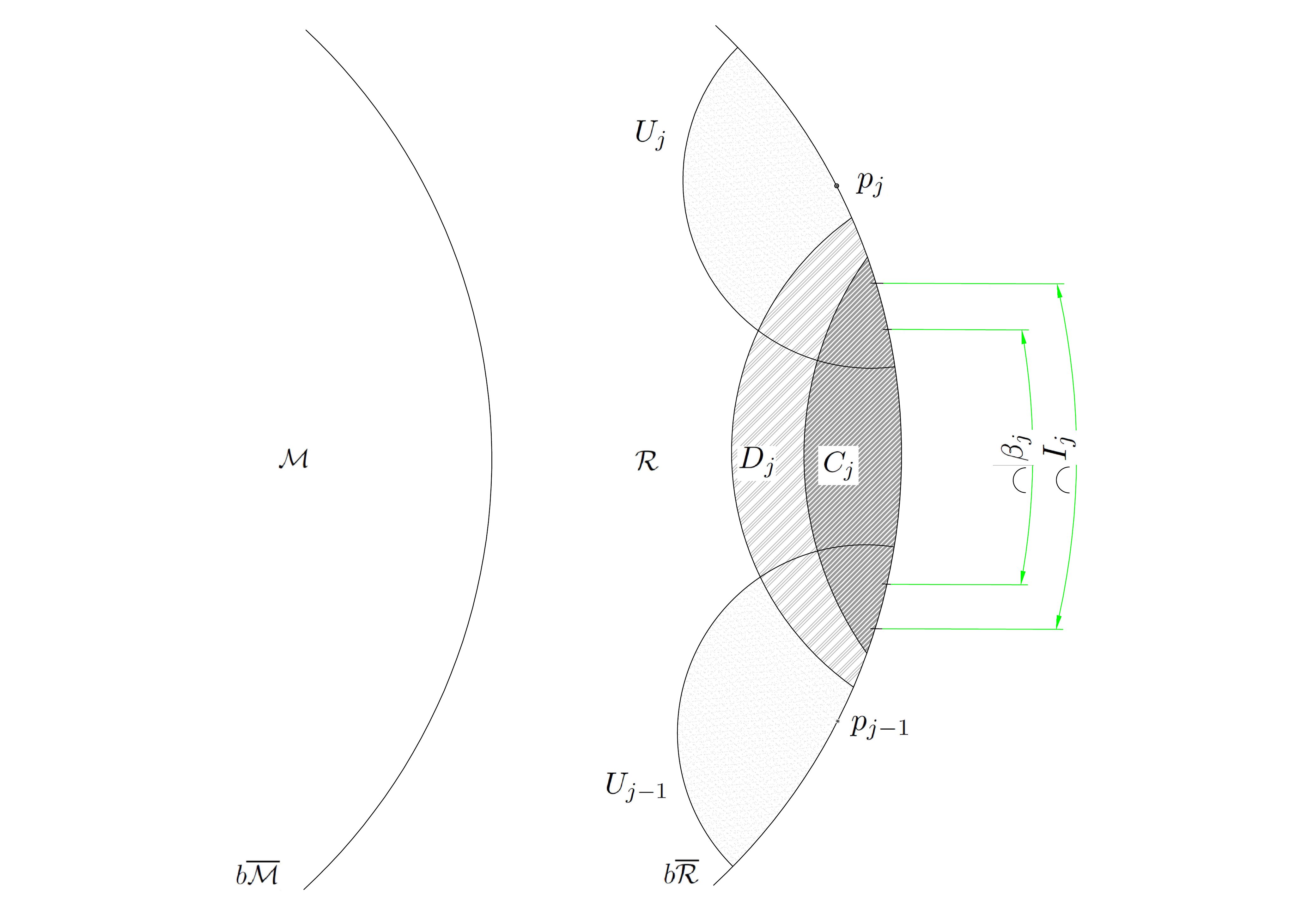}
\caption{Sets $U_{j-1}, U_j, C_j, D_j$ in $\ol{\Rcal} \setminus \ol{\Mcal}$ and arcs $\beta_j, I_j$ on $b\ol{\Rcal}$ constructed in step $3$.}
\label{fig:sets}
\end{center}
\end{figure}

We shall now use a solution of a suitable Riemann-Hilbert problem.
For each $j$, choose a unit vector $v_j \in \c^2$ perpendicular to the vector $f(p_j)-g(p_j) \neq 0$, and a continuous function $\mu_j \colon b\ol{D}_j \to [0,\eta]$ such that 
\begin{equation}\label{eq:mu_j}
	\begin{cases}
	\mu_j = \eta & \textrm{on } \beta_j, \\
	\mu_j = 0 & \text{on } b\ol{D}_j\setminus I_j.
	\end{cases}
\end{equation}
Next, take a continuous map $g_j \colon b\ol{D}_j \times \cd \to \c^2$ satisfying
\begin{itemize}
\item [\rm (i)] $g_j(p,\cdot) \in \Acal(\ol{\d})^2$, and

\smallskip
\item [\rm (ii)] $g_j(p,\xi)=f_0(p)+\mu_j(p) \cdot \xi \cdot v_j$ for all $p\in b\ol{D}_j, \ \xi \in \ol{\d}$.
\end{itemize}
Note that there exists a conformal diffeomorphism mapping $\ol{D}_j$ onto $\cd$.
In this situation, \cite[Lemma 6.1.1]{alarcon2021minimal} furnishes a map $h_j \in \Acal(\ol{D}_j)^2$, enjoying the following properties:
\begin{itemize}
\item[\rm (c1)] $\dist(h_j(p),g_j(p,\t)) < \varepsilon_2 \ \textrm{for all } p\in b\ol{D}_j$;

\smallskip
\item[\rm (c2)] $\dist(h_j(p),g_j(r(p),\cd)) < \varepsilon_2 \ \textrm{for all } p\in C_j$;

\smallskip
\item[\rm (c3)] $|h_j-f_0| < \varepsilon_2 \ \textrm{on } \ol{D}_j\setminus C_j$;

\smallskip
\item[\rm (c4)] $|\pi_j \circ h_j - \pi_j \circ f_0| < 2\varepsilon_2 \ \textrm{on } \ol{C}_j$.
\end{itemize}
Observe that in order to obtain (c4), we take into account assertions (ii), (b3), (c2), (c3) and that $v_j$ is orthogonal to $f(p_j)-g(p_j)$.

STEP 4 -- Gluing $f_0$ and $h_j$.
Define sets
\begin{equation*}
A = \Rcal \setminus \bigcup_{j\in \Z_l} \ol{C}_j \subset \Rcal \quad \textrm{and} \quad B = \bigcup_{j\in \Z_l} D_j \subset \Rcal.
\end{equation*}
Let $\ol{A}$ and $\ol{B}$ be their closures. Note that they form a Cartan pair for $\ol{\Rcal}$ in the sense of \cite[Definition 5.7.1]{Forstneric2017} and $(E\cup \Lambda) \cap (\ol{A \cap B})=\emptyset$.

We shall use a solution of the Cousin-I problem to obtain a global meromorphic map. 
At this step, we consider $f_0$ as a holomorphic map $\ol{\Rcal} \to \cp^2$, keeping in mind its properties.
Fix a number $\varepsilon_3>0$ to be specified later. By property (c3) and assuming that $\varepsilon_2>0$ is small enough, \cite[Proposition 5.8.1 and Lemma 5.8.2]{Forstneric2017} furnish a holomorphic map $\hat{f} \colon \ol{\Rcal} \to \cp^2$ which is $\varepsilon_3$-close to $f_0$ in $\Ccal^0(\ol{A})$ and $\varepsilon_3$-close to $h_j$ in $\Ccal^0(\ol{D}_j)$ for all $j\in \Z_l$, see also \cite[Corollary 2]{alarcon2013every}. 
Furthermore, we can choose $\hat{f}$ so that $\hat{f}-f_0$ vanishes to order $d$ on $E\cup \Lambda$, by addition to the cited lemma. (Note that the same proofs to \cite[Proposition 5.8.1]{Forstneric2017}, \cite[Lemma 5.8.2]{Forstneric2017} apply if one replaces $\c^N$ by $\cp^N$ in the target spaces of the corresponding maps, as is the case in our setting. So use of these results is justified.)
Using the latter and since $f_0$ is holomorphic outside of $E$, Hurwitz's theorem gives that $\hat{f} \colon \Rcal \setminus E \to \c^2$ is holomorphic, extending to a meromorphic map from $\Rcal$ to $\c^2$ having poles at all points in $E$, provided that $\varepsilon_3>0$ is sufficiently small.
By Mergelyan approximation, we may assume that $\hat{f}$ is a holomorphic map from a neighbourhood of $\ol{\Rcal} \setminus E$ to $\c^2$, which extends meromorphically to a neighbourhood of $\ol{\Rcal}$ with the polar set $E$. Applying a general position argument, we can assume that $\hat{f}$ is a meromorphic immersion on a neighbourhood of $\ol{\Rcal}$ with the polar set $E$.
(See \cite[Lemma 2, Corollary 2, Sect. 3.4]{alarcon2013every} for an analogous result in the case of holomorphic maps $\ol{\Rcal} \to \c^2$.)

STEP 5 -- Checking the properties of $\hat{f}$.

Fix $p\in b\ol{\Rcal}$. We claim that $|\hat{f}(p)-g(p)|<\sqrt{\delta^2+\eta^2}$. Let us consider two cases.
	
If $p\in b\ol{\Rcal}\cap \ol{A}$, then
\begin{equation*}
|\hat{f}(p)-g(p)| \leq |\hat{f}(p)-f_0(p)| + |f_0(p)-g(p)| < \varepsilon_3 + \delta < \sqrt{\delta^2+\eta^2}
\end{equation*}
for a suitable $\varepsilon_3>0$. (Note that the second inequality is obtained by the approximation in step 4 and a property of $f_0$.)

If $p\in b\ol{\Rcal}\cap \ol{B}$, there exists some $j \in \Z_l$ such that $p\in b\ol{\Rcal}\cap \ol{C}_j$. Pick a point $\xi \in b\ol{D}_j$. Then
\begin{equation*}
|\hat{f}(p)-g(p)| \leq |\hat{f}(p)-h_j(p)| + |h_j(p)-g_j(p,\xi)|+|g_j(p,\xi)-g(p)|.
\end{equation*}
We estimate the last summand as follows. Start with the point $f(p)$. 
Then $f_0(p)$ satisfies $|f(p)-f_0(p)|<\varepsilon_0$ and $f(p_j)$ satisfies $|f(p_j)-f_0(p)|<\varepsilon_0$ (recall the second property stated after the definition of $f_0$). 
Construct a complex line orthogonal to the vector $f(p_j)-g(p_j)$ and a parallel to it through the point $f_0(p)$. Then $g_j(p,\xi)$ lies on the latter line with $|f_0(p)-g_j(p,\xi)|=\mu_j(p)$, by condition (ii) in the definition of $g_j$. Let the point $T$ lie on the line we constructed first (passing through $f(p_j)$, orthogonal to $f(p_j)-g(p_j)$) and such that it satisfies $|T-f(p_j)|=|f_0(p)-g_j(p,\xi)|, \ |T-g_j(p,\xi)|=|f(p_j)-f_0(p)|$. By Pythagora's theorem, it follows that $|T-g(p_j)|=\sqrt{\mu_j^2(p)+|f(p_j)-g(p_j)|^2}$.
Hence 
\begin{eqnarray*}
	|g_j(p,\xi)-g(p)| & = & |(f_0(p)+\mu_j(p) \cdot \xi \cdot v_j) - g(p)| \\
	& \leq & |g_j(p,\xi)-T| + |T-g(p_j)| + |g(p_j)-g(p)| \\
	& \leq & |f(p_j)-f_0(p)| + \sqrt{\mu_j^2(p)+|f(p_j)-g(p_j)|^2} + |g(p_j)-g(p)| \\
	& < & \varepsilon_0 + \sqrt{\mu_j^2(p)+|f(p_j)-g(p_j)|^2} + \varepsilon_0.
\end{eqnarray*}
Note that to obtain this estimate, one applies the definition of $g_j$, triangular inequality, properties in the construction stated above and the fourth property in step 1.
Taking into account the approximation in step 4, (c1) and \eqref{eq:f-g} we conclude that
\begin{equation*}
|\hat{f}(p)-g(p)| < \varepsilon_3+\varepsilon_2+2\varepsilon_0 + \sqrt{\mu_j^2(p)+|f(p_j)-g(p_j)|^2} < \sqrt{\delta^2+\eta^2},
\end{equation*}
provided that $\varepsilon_0, \varepsilon_2, \varepsilon_3>0$ are sufficiently small.

We now prove $\dist_{\hat{f}}(p_0,b\ol{\Rcal})>\mu+\eta$.
Since $\ol{\Mcal} \subset \ol{A}$, it holds
\begin{equation*}
|\hat{f}-f| \leq |\hat{f}-f_0|+|f_0-f| < \varepsilon_3+\varepsilon_1 \quad \textrm{on } \ol{\Mcal}.
\end{equation*}
Moreover, \eqref{eq:dist-f-bM} says $\dist_f(p_0,b\ol{\Mcal}) > \mu$. Combining both observations, we infer that $\dist_{\hat{f}}(p_0,b\ol{\Mcal})>\mu$ for sufficiently small numbers $\varepsilon_1, \varepsilon_3>0$, so it suffices to prove that the distance $\dist_{\hat{f}}(b\ol{\Mcal},b\ol{\Rcal})$ is larger than a number as close to $\eta$ as desired.
Choose an arc $\gamma \subset \ol{\Rcal} \setminus \Mcal$ with the initial point on $b\ol{\Mcal}$ and the final point on $b\ol{\Rcal}$, but otherwise disjoint from $b\ol{\Mcal}$ and $b\ol{\Rcal}$. We shall estimate $\length(\hat{f}(\gamma))$. We look at two cases.

Firstly, assume that there exist some $j \in \Z_l$ such that $\gamma \cap U_j \neq \emptyset$ holds. Consider a subpath $\gamma' \subset \gamma$ with one endpoint on $b\ol{\Mcal}$ and the other endpoint in $\ol{U}_j$. Then
\begin{eqnarray*}
	\length(\hat{f}(\gamma')) & \geq & \length(\hat{f}(\gamma' \cap \ol{A})) + \length(\hat{f}(\gamma' \cap \ol{C}_j)) + \length(\hat{f}(\gamma' \cap \ol{C}_{j+1})) \\
	& \approx & \length(f_0(\gamma' \cap \ol{A})) + \length(h_j(\gamma' \cap \ol{C}_j)) \\
	& + & \length(h_{j+1}(\gamma' \cap \ol{C}_{j+1})) \\
	& \geq & \length((\pi_j \circ f_0)(\gamma' \cap \ol{A})) + \length((\pi_j \circ h_j)(\gamma' \cap \ol{C_j})) \\
	& + & \length((\pi_{j+1} \circ h_{j+1})(\gamma' \cap \ol{C}_{j+1})) \\
	& \stackrel{\text{\rm (c5)}}{\approx} & \length((\pi_j \circ f_0)(\gamma' \cap \ol{A})) + \length((\pi_j \circ f_0)(\gamma' \cap \ol{C}_j)) \\
	& + & \length((\pi_{j+1} \circ f_0)(\gamma' \cap \ol{C}_{j+1})) \\
	& \stackrel{\text{\eqref{eq:length-f0-Ja}}}{>} & \eta,
\end{eqnarray*}
hence the original path $\gamma$ satisfies $\length(\hat{f}(\gamma))>\eta$, as required. (Recall step 4 to obtain approximation of $\hat{f}$ by $f_0$ and $h_j$, respectively, in the second inequality.)

If, on the contrary, $\gamma \cap U_j = \emptyset$ for all $j$, then there exists a subarc $\gamma_1 \subset \gamma, \ \gamma_1 \subset \ol{D_j \setminus (U_{j-1} \cup U_j)}$, between points $p \in \ol{D_j \setminus (U_{j-1} \cup U_j \cup C_j)}$ and $q \in \beta_j$. We estimate as follows:
\begin{eqnarray}\label{eq:length-hatf}
	\length(\hat{f}(\gamma)) & \geq & \length(\hat{f}(\gamma_1)) \geq |\hat{f}(p)-\hat{f}(q)| \\ \nonumber
	& \geq & |\hat{f}(q)-f_0(q)| - |\hat{f}(p)-f_0(p)| \\ \nonumber
	& - & |f_0(p) - f_0\circ r(p)| - |f_0\circ r(p) - f_0(q)|.
\end{eqnarray}
Pick some $\xi \in \t$. Then
\begin{eqnarray*}
	|\hat{f}(q)-f_0(q)| & \geq & |g_j(q,\xi)-f_0(q)| - |h_j(q)-g_j(q,\xi)| - |\hat{f}(q)-h_j(q)| \\
	& = & \mu_j(q) \cdot |\xi| - |h_j(q)-g_j(q,\xi)| - |\hat{f}(q)-h_j(q)| \\
	& > & \eta - \varepsilon_2 - \varepsilon_3
\end{eqnarray*}
holds by \eqref{eq:mu_j}, the definition of $g_j$ and (c1). Using the above estimate and (b3) in the inequality \eqref{eq:length-hatf}, we get
\begin{equation*}
\length(\hat{f}(\gamma)) > (\eta -\varepsilon_2-\varepsilon_3)-\varepsilon_3-2\varepsilon_2 > \eta-3\varepsilon_2-2\varepsilon_3.
\end{equation*}
Hence $\dist_{\hat{f}}(b\ol{\Mcal},b\ol{\Rcal}) > \eta-3\varepsilon_2-2\varepsilon_3$ and $\dist_{\hat{f}}(p_0,b\ol{\Rcal}) > \mu + \eta$ is guaranteed by sufficiently small choice of numbers $\varepsilon_2, \varepsilon_3 > 0$.

To prove $|\hat{f}-f|(p)<\varepsilon \ \textrm{for all } p\in \ol{\Mcal}$, observe that
	\begin{equation*}
	|\hat{f}-f|(p) \leq |\hat{f}-f_0|(p) + |f_0-f|(p) < \varepsilon_3 + \varepsilon_1
	\end{equation*}
holds by step 4 and properties of $f_0$. Provided that $\varepsilon_1, \varepsilon_3>0$ are small enough, we get the required estimate.

Note that the map $\hat{f}-f$ is holomorphic on $\Rcal$ and it vanishes to order $d$ at every point in $E\cup \Lambda$ by the construction of $\hat{f}$ and approximations at every step of the procedure. More precisely, the difference $\hat{f}-f$ satisfies the required by step 4, whereas the same claim for $f_0-f$ is provided by the definition of $f_0$ in \eqref{eq:f_0} and step 2.
Recall also that $E \cup \Lambda \subset \Mcal$ and $\ol{D}_j \subset \ol{\Rcal}\setminus \ol{\Mcal}$ for all $j$, thus interpolation on the set $E \cup \Lambda$ from step 2 is not affected in step 3.
This concludes the proof.
\end{proof}

By an inductive application of Lemma \ref{lem:dist} and using a standard general position argument in the final step, we obtain Lemma \ref{lem:main}. The proof goes as follows.

%
%
\begin{proof}[Proof of Lemma \ref{lem:main}]
Choose numbers $\eta_0$ and $\delta_0$ such that $0<\eta_0<\dist_{f}(p_0,b\ol{\Rcal})$ and $0<\delta_0<\varepsilon$.
Set 
\begin{equation*}\label{eq:c}
c = \sqrt{\frac{6(\varepsilon^2-\delta_0^2)}{\pi^2}} >0
\end{equation*}
and for $j \in \N$ recursively define sequences
\[ \eta_{j} = \eta_{j-1} + \frac{c}{j}>0, \quad \delta_{j} = \sqrt{\delta_{j-1}^2 + \frac{c^2}{j^2}}>0. \]
It follows that
\begin{equation}\label{eq:nj}
\lim_{j \to \infty} \eta_{j} = \eta_0 + \sum_{j=1}^{\infty}\frac{c}{j} = +\infty
\end{equation}
and
\begin{equation}\label{eq:deltaj}
\lim_{j \to \infty} \delta_{j}^2 = \delta_0^2 + \sum_{j=1}^{\infty}\frac{c^2}{j^2} = \delta_0^2 + \frac{c^2 \pi^2}{6} = \varepsilon^2, \ \textrm{hence } \lim_{j \to \infty} \delta_{j} = \varepsilon.
\end{equation}
We shall inductively construct a sequence of maps $f_{j} \in \IE(\Rcal,\c^2)$ meeting the following conditions  for every $j \in \z_{\geq0}$:
\begin{itemize}
\item[\rm (i$_j$)] $|f_{j}(p)-f(p)|<\delta_j$ for all $p \in b\ol{\Rcal}$;

\smallskip
\item[\rm (ii$_j$)] $\dist_{f_j}(p_0, b\ol{\Rcal}) > \eta_j$;

\smallskip
\item[\rm (iii$_j$)] $f_j-f$ is holomorphic on $\Rcal$ and vanishes to order $d$ at every point in $E \cup \Lambda$. 
\end{itemize}
Set $f_0 =f$ for the basis of induction. Now assume we have already found maps $f_k$ for all $k\leq j$ for some $j\in \Z_{\geq 0}$. Applying Lemma \ref{lem:dist} to the data 
\[ \Rcal, \ f=f_j, \ E, \ \Lambda, \ p_0, \ d, \ g=f|_{b\ol{\Rcal}}, \ \delta=\delta_j, \ \mu=\eta_j, \ \eta=\frac{c}{j+1}, \ \varepsilon, \]
we get a map $\hat{f} = f_{j+1} \in \IE(\Rcal,\c^2)$. 
Clearly, it satisfies:
\begin{itemize}
\item[\rm (i$_{j+1}$)] $|f_{j+1}(p)-f(p)|< \sqrt{\delta_j^2 + \frac{c^2}{(j+1)^2}} = \delta_{j+1}$ for all $p \in b\ol{\Rcal}$;

\smallskip
\item[\rm (ii$_{j+1}$)] $\dist_{f_{j+1}}(p_0, b\ol{\Rcal}) > \eta_j + \frac{c}{j+1} = \eta_{j+1}$;

\smallskip
\item[\rm (iii$_{j+1}$)] $f_{j+1}-f$ is holomorphic on $\Rcal$ and vanishes to order $d$ at every point in $E \cup \Lambda$. 
\end{itemize}
This closes the induction.
Properties $(\textrm{i}_j)$ and \eqref{eq:deltaj} imply $|f_j-f|(p)<\varepsilon \ \textrm{for all } p\in b\ol{\Rcal} \ \textrm{and} \ j\in \N$, and by the maximum principle, the same inequality holds for all $p \in \ol{\Rcal}$.
Moreover, $\textrm{(ii}_j)$ and \eqref{eq:nj} guarantee that $\dist_{f_j}(p_0, b\ol{\Rcal})>\eta_j>\tau$ for large enough $j\in \N$.
Choosing $j\in \N$ sufficiently large such that the above is satisfied and applying a general position argument to $f_{j}$ in order to ensure condition (iv) in the statement of the lemma (note that $b\ol{\Rcal}$ is of real dimension $1$ while $\c^2$ is of real dimension $4$), we obtain a map $\hat{f}$ which meets the conclusions of Lemma \ref{lem:main}.
\end{proof}
%
%
%


\section{Proof of Theorem \ref{th:main}}\label{sec:proof-main}

\noindent
For simplicity of exposition, we shall assume $n=2$; the general case follows from the same argument. At the end of the proof we explain how to ensure embeddedness of $\hat{h}$ on $\ol{\Rcal}\setminus E$ for $n\geq 3$ (see Remark \ref{rem:n>2}).

Let us assume that $\Rcal$ is a bordered domain in an open Riemann surface $\wh{\Rcal}$ such that $\Rcal \in \Bo(\wh{\Rcal})$.
Choose bordered domains $\Rcal_0, S_0 \subset \Rcal$, $\Rcal_0 \in \Bo(\Rcal)$, with $E \cup \Lambda \subset \Rcal_0$, $E \subset S_0, \ S_0 \cap \Lambda = \emptyset, \ \ol{S}_0 \subset \Rcal_0$ (i.e., $S_0$ is a pairwise disjoint union of discs centered at the points in $E$), point $p_0 \in \Rcal_0 \setminus (S_0 \cup \Lambda)$ and positive integer $\tau_0 \in \N$.
Let $d \colon \wh{\Rcal} \times \wh{\Rcal} \to \r$ be a distance function on the Riemannian surface $(\wh{\Rcal},g)$, where $g$ is a Riemannian metric on $\wh{\Rcal}$.
Set $\varepsilon_0 = \varepsilon$ and $h_0 = h$. By Mergelyan's theorem and general position, we may assume that $h_0 \in \IE(\Rcal,\c^2)$ (see Definition \ref{def:IE}) and by a standard general position argument, assume that the restriction $h_0|_{b\ol{\Rcal}}$ is an embedding.
We will inductively construct a sequence $\left\{ \Pi_j = \{\Rcal_j, S_j, h_j, \varepsilon_j, \tau_j \} \right\}_{j=1}^{\infty}$ of bordered domains $\Rcal_j, S_j \subset \Rcal$ such that $\Rcal_j \in \Bo(\Rcal)$, $\ol{S}_j \subset \Rcal_j$, maps $h_j \in \IE(\Rcal,\c^2)$, positive numbers $\varepsilon_j>0$ and integers $\tau_j>j$, satisfying the following properties for all $j \in \n$.
\begin{itemize}
\item[\rm (1$_j$)] 
	$\ol{\Rcal}_{j-1} \subset \Rcal_j, \ \ol{S}_j \subset S_{j-1}$ and
	\begin{equation*}
	\max \left\{\max_{p\in b\ol{\Rcal}_j} \{d(p,b\ol{\Rcal})\}, \max_{i=1, \dots, \mgot} \{d(e_i,bS_{j}^{i})\} \right\} < \frac{1}{j},
	\end{equation*}
	where $E=\{e_i\}_{i=1}^{\mgot} \ (\mgot \in \n)$ and $S_{j}^{i}$ is a disc centered at $e_i$ such that $\bigcup_{i=1}^{\mgot}S_{j}^{i}=S_j$ is a pairwise disjoint union.

\smallskip
\item[\rm (2$_j$)] 
	$\max \left\{||h_j-h_{j-1}||_{0,\ol{\Rcal}}, ||dh_j-dh_{j-1}||_{0,\ol{\Rcal}_{j-1}} \right\} < \varepsilon_j$.

\smallskip
\item[\rm (3$_j$)] 
	$\dist_{h_j}(p_0,b\ol{\Rcal}_k)>k$ for all $k\in \{0,\dots, j\}$.

\smallskip
\item[\rm (4$_j$)] 
	$h-h_j$ is holomorphic on $\Rcal_j$, vanishing to any given finite order on $E \cup \Lambda$.

\smallskip
\item[\rm (5$_j$)] 
	$h_j|_{b\ol{\Rcal}}$ is an embedding.

\smallskip
\item[\rm (6$_j$)] 
	$0 < \varepsilon_{j} < \min \left\{ \frac{\varepsilon_{j-1}}{2}, \frac{1}{\tau_{j}^{j}}, \frac{1}{2j^2} \inf\left\{|h_{j-1}(p)-h_{j-1}(q)| \colon p,q\in b\ol{\Rcal}, d(p,q)>\frac{1}{j} \right\} \right\}$,
	and such that if $g \colon \Rcal \setminus E \to \c^2$ is holomorphic and $|g-h_j|<\varepsilon_j$ on $\ol{\Rcal}_j \setminus S_j$, then $g$ is an immersion on $\Rcal_{j-1} \setminus S_{j-1}$ and $\dist_g(p_0,b\ol{\Rcal}_{k-1})>k-1$ for all $k\in \{0,\dots,j-1\}$.

\smallskip
\item[\rm (7$_j$)]
	 $\tau_j>\tau_{j-1}$, and there is a set $A_{j} \subset h_j(b\ol{\Rcal})$ of cardinality $\tau_{j}^{j+1}$ so that
	\begin{equation*}
	\max \{ \dist(p,A_{j}) \colon p\in h_j(b\ol{\Rcal}) \} < \frac{1}{\tau_{j}^{j}}.
	\end{equation*}
\end{itemize}
Firstly, $\Pi_0$ satisfies conditions $(3_0)$--$(5_0)$ and the second part of $(7_0)$ for large enough $\tau_0$, whereas the others are void.
Now assume we have already constructed sequences $\Pi_0, \dots, \Pi_{j-1}$ for some $j\in \n$. Since $h_k, \ k\in \{0,\dots,j-1\}$, are embeddings on $b\ol{\Rcal}$, we can find a number $\varepsilon_j>0$ satisfying the first part of $(6_j)$. Lemma \ref{lem:main} furnishes a map $h_j \in \IE(\Rcal,\c^2)$ such that
\begin{itemize}
\item $||h_j-h_{j-1}||_{0,\ol{\Rcal}}<\varepsilon_j$, which implies approximation in $\Ccal^1(\Rcal_{j-1})$ topology, i.e., $(2_j)$;

\smallskip
\item $\dist_{h_j}(p_0,b\ol{\Rcal})$ is arbitrarily big;

\smallskip
\item $h_j-h_{j-1}$ is holomorphic on $\Rcal$, vanishing to any given finite order on $E \cup \Lambda$, hence by $(4_{j-1})$, the same holds for the map $h-h_j$, i.e., $(4_j)$;

\smallskip 
\item $h_j|_{b\ol{\Rcal}}$ is an embedding, i.e., $(5_j)$.
\end{itemize}
Letting $\tau_j \in \n$ so that
\begin{equation*}
\tau_j > \tau_{j-1} + \length(h_j(b\ol{\Rcal}))
\end{equation*}
holds, it satisfies property $(7_j)$.
Choosing domains $\Rcal_j$ sufficiently big and $S_j$ sufficiently small, we fulfill conditions $(1_j), \ (3_j) \ \textrm{and } (6_j)$. (Note that in order to obtain $(6_j)$, we require $\varepsilon_j>0$ to be small enough, and then apply Cauchy estimates.)
This closes the induction and finishes the construction of $\{\Pi_j\}_j$. Moreover, we can ensure that $\Rcal=\bigcup_{j=0}^{\infty}\Rcal_j$ and $E=\bigcap_{j=0}^{\infty}S_j$ by a suitable choice of each pair of domains $\Rcal_j$ and $S_j$.

Taking into account properties $(1_j), \ (2_j) \ \textrm{and } (6_j)$, we infer that the sequence of maps $\{h_j \colon \ol{\Rcal} \setminus E \to \c^2\}_j$, that are holomorphic and extend to meromorphic maps on $\ol{\Rcal}$, converges on $\ol{\Rcal} \setminus E$ to a continuous limit map
\begin{equation*}
\hat{h} = \lim_{j\to \infty}{h_j} \colon \ol{\Rcal} \setminus E \to \c^2,
\end{equation*}
which is meromorphic on $\Rcal$.
%
Let us check its properties.

It follows from the construction and condition $(4_j)$ that the map $\hat{h}-h$ is holomorphic on $\Rcal$, vanishing to any given finite order on $E \cup \Lambda$.

Fix some $p\in \ol{\Rcal}$ and $j\in \n$. Since
\begin{equation} \label{eq:hath-hj<epsilonj}
|\hat{h}-h_j|(p) \leq \sum_{k=j}^{\infty}|h_{k+1}-h_k|(p)  \stackrel{\text{\rm (2$_{k+1}$)}}{<} \sum_{k=j}^{\infty}\varepsilon_{k+1}<2\varepsilon_{j+1}<\varepsilon_j,
\end{equation}
$\hat{h}$ is an immersion on $\Rcal_{j-1} \setminus S_{j-1}$ and $\dist_{\hat{h}}(p_0,b\ol{\Rcal}_{j-1})>j-1$ due to condition $(6_j)$. This holds for all $j\in \n$, implying that $\hat{h}$ is a complete immersion on $\Rcal \setminus E$.
Moreover, the above estimate for $j=0$ gives
\begin{equation*}
|\hat{h}-h|(p)=|\hat{h}-h_0|(p)<\varepsilon_0=\varepsilon, \quad p\in \ol{\Rcal},
\end{equation*}
proving assertion (i) in the theorem.

Let $p\neq q$ be two distinct points on $b\ol{\Rcal}$ and take $j_0 \in \n$ such that $d(p,q)>1/j_0$. Pick some $j>j_0$. Then
\begin{eqnarray*}
|h_{j-1}(p)-h_{j-1}(q)| & \leq & |h_j(p)-h_{j-1}(p)| + |h_j(q)-h_{j-1}(q)| + |h_j(p)-h_j(q)| \\
	&  \stackrel{\text{\rm (2$_j$)}}{<} & 2\varepsilon_j + |h_j(p)-h_j(q)| \\
	&  \stackrel{\text{\rm (6$_j$)}}{<} & 2 \cdot \frac{1}{2j^2} \left|h_{j-1}(p)-h_{j-1}(q)\right| + |h_j(p)-h_j(q)|,
\end{eqnarray*}
thus
\begin{equation*}
|h_j(p)-h_j(q)| > \left(1-\frac{1}{j^2}\right) |h_{j-1}(p)-h_{j-1}(q)|
\end{equation*}
and
\begin{equation*}
|h_{j_0+l}(p)-h_{j_0+l}(q)| > \prod_{j=j_0+1}^{j_0+l} \left(1-\frac{1}{j^2}\right) |h_{j_0}(p)-h_{j_0}(q)|, \quad l\in \n.
\end{equation*}
In the limit as $l\to \infty$, 
\begin{equation*}
|\hat{h}(p)-\hat{h}(q)| \geq \frac{1}{2} \left|h_{j_0}(p)-h_{j_0}(q)\right| \stackrel{\text{\rm (5$_{j_0}$)}}{>} 0,
\end{equation*}
proving that the restriction $\hat{h}|_{b\ol{\Rcal}}$ is an embedding.

We claim that the union of boundary Jordan curves $\hat{h}(b\ol{\Rcal})$ has Hausdorff dimension one. (We refer to \cite[Sect. 2.3]{morgan2016geometric} for the basics on Hausdorff measure and dimension.) Indeed, properties $(6_j), \ (7_j)$ and \eqref{eq:hath-hj<epsilonj} imply
\begin{equation*}
\max \{ \dist(p,A_{j}) \colon p\in \hat{h}(b\ol{\Rcal}) \} <\frac{2}{\tau_{j}^{j}}, \quad j\in \n.
\end{equation*}
By $(7_j)$, the set $A_{j} \subset h_j(b\ol{\Rcal})$ consists of $\tau_{j}^{j+1}$ points for all $j\in \n$. Denote them by $a_t$ for $t \in \{1, \dots, \tau_{j}^{j+1} \}$, and consider balls $\b_t(a_t,4/\tau_{j}^{j}) \subset \c^2$ centered at points $a_t$ and of radii $4/\tau_{j}^{j}$. Their union covers $\hat{h}(b\ol{\Rcal})$. Since
\begin{equation*}
\tau_{j}^{j+1} \cdot \left(\frac{4}{2\tau_{j}^{j}} \right)^{1+1/j} = 2^{1+1/j} \leq 4 < +\infty, \quad j\in \n,
\end{equation*}
the $(1+1/j)$-dimensional Hausdorff measure of $\hat{h}(b\ol{\Rcal})$ is finite for all $j\in \n$. This implies that the Hausdorff measure of $\hat{h}(b\ol{\Rcal})$, denoted by $\Hcal^1(\hat{h}(b\ol{\Rcal}))$, is finite, thereby its Hausdorff dimension is bounded above by one.
However, $\hat{h}(b\ol{\Rcal})$ is a pairwise disjoint union of Jordan curves, so it has Hausdorff dimension bounded below by one. We conclude that the Hausdorff dimension of $\hat{h}(b\ol{\Rcal})$ equals one, as required, thus proving the theorem. (See \cite[Sect. 2.1]{martin2007jordan} for a similar argument.)

\begin{remark} \label{rem:n>2}
An analogous proof applies in the case of maps $h\colon \ol{\Rcal} \setminus E \to \c^n$ for $n\geq 3$.
At the $j$-th step of induction, proceeding as before, we may additionally assume that $h_j$ is an embedding on the compact set $\ol{\Rcal} \setminus S_j$, applying general position.
Moreover, one shall replace the first part of condition $(6_j)$ by asking
\begin{equation*}
0 < \varepsilon_{j} < \min \left\{ \frac{\varepsilon_{j-1}}{2}, \frac{1}{\tau_{j}^{j}}, \frac{1}{2j^2} \inf\left\{|h_{j-1}(p)-h_{j-1}(q)| \colon p,q\in \ol{\Rcal} \setminus S_{j-1}, d(p,q)>\frac{1}{j} \right\} \right\}.
\end{equation*}
This grants that the limit map $\hat h \colon \ol{\Rcal}\setminus E \to \C^n$ is a topological embedding.
\end{remark}


\section{Removing countably many discs} \label{sec:countably}

\noindent
In this section we prove a generalized version of Theorem \ref{th:main} -- we ask the domain to be a complement in a compact Riemann surface of a countable union of pairwise disjoint smoothly bounded closed discs. In the case when the union is finite, the boundary of such domain consists of finitely many smooth closed Jordan curves, hence the domain is a bordered Riemann surface. This situation is considered in Theorem \ref{th:main}. The precise formulation of the generalized theorem is the following.

%
%
\begin{theorem} \label{th:countably}
Let $\Rcal$ be a compact Riemann surface and let $\Mcal = \Rcal \setminus \bigcup_{i=0}^{\infty}D_i$ be a domain in $\Rcal$ whose complement is a countable union of pairwise disjoint smoothly bounded closed discs $D_i$. Let $E, \Lambda \subset \Mcal$ be finite disjoint subsets. Set $\Mcal_0 = \Rcal \setminus \mathring D_0$.
If $h\colon \Mcal_0 \setminus E \to \c^n$ ($n\geq 2$) is a continuous map which is meromorphic on $\mathring \Mcal_0$ having poles at all points in $E$, then for any positive number $\varepsilon >0$ there exists a continuous map $\hat{h} \colon \ol{\Mcal} \setminus E \to \c^n$, which is a meromorphic immersion on $\Mcal$, satisfying the following properties:
\begin{itemize}
\item[\rm (i)] the difference $h-\hat{h}$ is holomorphic on $\Mcal$, vanishing to any given finite order on $E \cup \Lambda$;

\smallskip
\item[\rm (ii)] $|h-\hat{h}|(p) < \varepsilon \ \textrm{for all }p \in \ol{\Mcal}$;

\smallskip
\item[\rm (iii)] $\hat{h} \colon \Mcal \setminus E \to \c^n$ is complete;

\smallskip
\item[\rm (iv)] $\hat{h}|_{b\ol{\Mcal}} \colon b\ol{\Mcal} \to \c^n$ is a topological embedding and $\hat{h}(b\ol{\Mcal}) = \bigcup_{i=0}^{\infty}\hat{h}(bD_i)$ consists of a pairwise disjoint union of Jordan curves we may choose of Hausdorff dimension one.
\end{itemize}
Furthermore, if $n\geq 3$, then $\hat{h} \colon \ol{\Mcal} \setminus E \to \c^n$ can be chosen an embedding with these properties, provided that $h|_{\Lambda}$ is injective.
\end{theorem}

The key in the proof of Theorem \ref{th:countably} is again Lemma \ref{lem:main}, however, the induction process is slightly different to that in the previous result. 

%
%
\begin{proof}[Proof of Theorem \ref{th:countably}]
We shall explain the proof in case of $n=2$. For $n\geq 3$, embeddedness of $\hat{h}$ on $\ol{\Mcal} \setminus E$ follows analogously, as in the proof of Theorem \ref{th:main} (see Remark \ref{rem:n>2}).

Assume that $\Rcal, \Mcal, \Mcal_0, E, \Lambda$ are as in the statement of the theorem. Denote the points in $E$ by $E=\{e_k\}_{k=1}^{\mgot} \ (\mgot \in \n)$. For $i\in \Z_{\geq0}$, we define
\begin{equation*}
\Mcal_i = \Rcal \setminus \bigcup_{j=0}^{i}\mathring D_j;
\end{equation*}
it is a compact bordered Riemann surface with boundary $b\Mcal_i = \bigcup_{j=0}^{i}bD_j$, satisfying
\begin{equation*}
\Mcal_0 \supset \Mcal_1 \supset \Mcal_2 \supset \dots \supset \bigcap_{i=0}^{\infty}\Mcal_i = \ol{\Mcal}.
\end{equation*}
Note that $E \cup \Lambda \subset \mathring \Mcal_i$ for all $i$.
Fix a normal exhaustion $\Kcal_0 \subset \Kcal_1 \subset \dots \subset \bigcup_{i=0}^{\infty}\Kcal_i = \Mcal$ of $\Mcal$, such that $E \cup \Lambda \subset \mathring \Kcal_0$. We let $S_0 \subset \Mcal$ be a bordered domain consisting of a finite union of pairwise disjoint discs with $E \subset S_0, \ \Lambda \cap S_0 = \emptyset, \ \ol{S}_0 \subset \Kcal_0$ (see proof of Theorem \ref{th:main}).
Choose a Riemannian distance function $d$ on $\Rcal$, point $p_0 \in \Mcal \setminus (S_0 \cup \Lambda)$ and positive integer $\tau_0 \in \N$.  
Set $h_0=h$ and $\varepsilon_0 = \varepsilon /2$.
By Mergelyan's theorem and arguing as in step 2 in the proof of Lemma \ref{lem:dist}, we may assume that $h_0 \in \IE(\mathring \Mcal_0, \c^2)$, and by a standard general position argument, assume that the restriction $h_0|_{b\Mcal_0}$ is an embedding.
We shall proceed by induction, applying Lemma \ref{lem:main} and \cite[Lemma 2.2]{alarcon2021calabi}, to construct a sequence $\{ \Pi_i = \{h_i, S_i, \varepsilon_i, \tau_i \}\}_{i=1}^{\infty}$ of maps $h_i \in \IE(\mathring \Mcal_i, \c^2)$, bordered domains $S_i \subset \Mcal$, positive numbers $\varepsilon_i>0$ and integers $\tau_i>i$, satisfying the following properties for all $i\in \n$.
\begin{itemize}
\item[\rm (1$_i$)] 
	$\ol{S}_i \subset S_{i-1}$ and $\max \{d(e_k,bS_{i}^{k}) \colon k=1,\dots, \mgot \} < \frac{1}{i}$,
	where $S_{i}^{k}$ is a disc centered at $e_k$ such that $\bigcup_{k=1}^{\mgot}S_{i}^{k}=S_i$ is a pairwise disjoint union (see proof of Theorem \ref{th:main}).

\smallskip
\item[\rm (2$_i$)] 
	$h-h_i$ is holomorphic on $\mathring \Mcal_i$, vanishing to any given finite order on $E \cup \Lambda$.

\smallskip
\item[\rm (3$_i$)] 
	$||h_i-h_{i-1}||_{0,\Mcal_i} < \varepsilon_{i-1}$.

\smallskip
\item[\rm (4$_i$)] 
	$\dist_{h_i}(p_0,b\Mcal_i)>i$.

\smallskip
\item[\rm (5$_i$)] 
	$h_i|_{b\Mcal_i}$ is an embedding.

\item[\rm (6$_i$)] 
	For every continuous map $g\colon \Mcal_i \setminus E \to \c^2$ such that $g-h_i$ is continuous on $\Mcal_i$ and satisfying $||g-h_i||_{0,\Mcal_i}<2\varepsilon_i$, we have
	\begin{equation*}
	\inf \{ \length(g\circ \gamma) \colon \gamma \in \Ccal_{qd}(\Mcal_i,p_0) \} \geq \dist_{h_i}(p_0,b\Mcal_i)-1 > i-1.
	\end{equation*}

\smallskip
\item[\rm (7$_i$)] 
	$0 < \varepsilon_{i} < \min \left\{ \frac{\varepsilon_{i-1}}{2}, \frac{1}{2\tau_{i}^{i}}, \frac{1}{2i^2} \inf\left\{|h_{i}(p)-h_{i}(q)| \colon p,q\in b\Mcal_i, d(p,q)>\frac{1}{i} \right\} \right\}$,
	and such that if $\tilde{g} \colon \Mcal \setminus E \to \c^2$ is holomorphic with poles in $E$ and $|\tilde{g}-h_i|<2\varepsilon_i$ on $\Kcal_i \setminus S_i$, then $\tilde{g}$ is an immersion on $\Kcal_{i-1} \setminus S_{i-1}$.

\smallskip
\item[\rm (8$_i$)]
	$\tau_i>\tau_{i-1}$, and for each nonnegative integer $j\leq i$ there is a set $A_{i,j}\subset h_i(bD_j)$ of cardinality $\tau_{i}^{i+1}$ so that
	\begin{equation*}
	\max \{ \dist(p,A_{i,j}) \colon p\in h_i(bD_j) \} < \frac{1}{\tau_{i}^{i}}.
	\end{equation*}
\end{itemize}
Firstly, $\Pi_0$ satisfies conditions $(2_0), \ (4_0), \ (5_0)$ and the second part of $(8_0)$, provided that $\tau_0$ is large enough. Other conditions are void.
Now assume we have constructed sequences $\Pi_0, \dots, \Pi_{i-1}$ for some $i\in \n$.
Applying Lemma \ref{lem:main} to $h_{i-1}|_{\Mcal_i}$, we obtain a map $h_i \in \IE(\mathring \Mcal_i,\c^2)$ such that $(2_i)$--$(5_i)$ hold. Taking $\tau_i \in \n$ with
\begin{equation*}
\tau_i > \tau_{i-1} + \sum_{j=0}^{i} \length(h_i(bD_j))
\end{equation*}
fulfills condition $(8_i)$. Since $h_i|_{b\Mcal_i}$ is an embedding by $(5_i)$, we may choose a number $\varepsilon_i>0$ satisfying the first part of $(7_i)$. Passing to a smaller $\varepsilon_i>0$ if necesssary, we can ensure $(6_i)$. Finally, choose $S_i$ small enough to satisfy $(1_i)$ and the second part of $(7_i)$. In addition, we may guarantee $E=\bigcap_{i=0}^{\infty}S_i$.

Note that $(6_i)$ is obtained by \cite[Lemma 2.2]{alarcon2021calabi}. 
Let $\gamma$ be a path in $\Mcal_i$ connecting $p_0$ with some point in $E$. Since $E$ is the polar set of $h_i$ (for all $i$), the path $h_i \circ \gamma$ has infinite Euclidean length. So, we can consider a set $\Mcal_{i}^{*} \subset \Mcal_i$, i.e., a smoothly bounded compact domain which is the complement in $\Mcal_i$ of small open neighbourhoods of points in $E$. We choose it so that
\begin{equation*}
\dist_{h_i}(p_0,b\Mcal_{i}^{*}) = \dist_{h_i}(p_0,b\Mcal_i).
\end{equation*}
We infer that
\begin{eqnarray*}
\inf \{ \length(g\circ \gamma) \colon \gamma \in \Ccal_{qd}(\Mcal_i,p_0) \} & \geq & 
		\inf \{ \length(g\circ \gamma) \colon \gamma \in \Ccal_{qd}(\Mcal_{i}^{*},p_0) \} \\
	& \geq & \dist_{h_i}(p_0,b\Mcal_{i}^{*})-1 \\
	& = & \dist_{h_i}(p_0,b\Mcal_i)-1,
\end{eqnarray*}
where the second inequality is obtained by applying \cite[Lemma 2.2]{alarcon2021calabi} to $\Mcal_{i}^{*}$ and $g$ as in $(6_i)$.

By $(3_i)$ and $(7_i)$, it follows that the sequence of maps $\{h_i \colon \Mcal_i \setminus E \to \c^2\}_i$, that are holomorphic and extend to meromorphic maps on $\Mcal_i$ with poles in $E$, converges on $\ol{\Mcal} \setminus E$ to a continuous limit map
\begin{equation*}
\hat{h} = \lim_{i\to \infty}h_i \colon \ol{\Mcal}\setminus E \to \c^2,
\end{equation*}
which is meromorphic on $\Mcal$ with poles at all points in $E$.

Fix some $p\in \ol{\Mcal}$ and $i\in \n$. Then
\begin{equation} \label{eq:hath-hj<2epsilonj}
|\hat{h}-h_{i}|(p) \leq \sum_{j=i}^{\infty}|h_{j+1}-h_j|(p) \stackrel{\text{\rm (3$_{j+1}$)}}{<} \sum_{j=i}^{\infty}\varepsilon_j < 2\varepsilon_i
\end{equation}
and by $(7_i)$, $\hat{h}$ is an immersion on $\Kcal_{i-1}\setminus S_{i-1}$. As this holds for all $i\in \n$, $\hat{h}$ is an immersion on $\Mcal \setminus E$.
Moreover,
\begin{equation*}
|\hat{h}-h|(p) < 2\varepsilon_0 = \varepsilon, \quad p\in \ol{\Mcal},
\end{equation*}
proving property (ii) in the theorem.
Assertion (i) is implied by conditions $(2_i)$ for all $i$.

To check the completeness, let us first extend $\hat{h}$ to a continuous map $\hat{h}\colon \Mcal_i \setminus E \to \c^2$, preserving its properties and satisfying \eqref{eq:hath-hj<2epsilonj} for all $p\in \Mcal_{i}$ (and $i$). 
Let $\gamma \colon [0,1) \to \Mcal\setminus E$ be a divergent path with the initial point $\gamma(0) = p_0$. Take an increasing sequence of numbers $0<t_1<t_2< \dots <1$ satisfying $\lim_{i\to \infty}t_i=1$ and $\lim_{i\to \infty}\gamma(t_i) = p\in b\ol{\Mcal}$. Clearly, this means $p\in bD_{i'}$ for some $i' \in \n$, implying that $p\in b\Mcal_i$ for all $i\geq i'$; thus we obtain a quasidivergent path $\gamma$ in $\Mcal_i \setminus E$ for all $i\geq i'$.
Properties \eqref{eq:hath-hj<2epsilonj} and $(6_i)$ give
\begin{equation*}
\length(\hat{h}\circ \gamma) \geq \dist_{h_i}(p_0,b\Mcal_i)-1 > i-1, \quad i\geq i'.
\end{equation*}
In the limit as $i\to \infty$, it follows that $\length(\hat{h}\circ \gamma) = +\infty$, proving condition (iii).

Arguing as in the proof of Theorem \ref{th:main}, it turns out that the restriction $\hat{h}|_{b\ol{\Mcal}}$ is an embedding (recall $\ol{\Mcal} \subset \Mcal_i$, $(3_i)$, $(5_i)$ and $(7_i)$).

Finally, the fact that each boundary Jordan curve $\hat{h}(bD_j)$, $j\in \z_{\geq 0}$, can be chosen of Hausdorff dimension one is seen as in the proof of Theorem \ref{th:main}, by using properties $(7_i), \ (8_i)$ and \eqref{eq:hath-hj<2epsilonj}.
\end{proof}
%
%
%


\subsection*{Acknowledgements}
This research was partially supported by the State Research Agency (AEI) via the grant no.\ PID2020-117868GB-I00, funded by MCIN/AEI/10.13039/501100011033/, Spain.
The author would like to thank Franc Forstneri{\v{c}} for his helpful remarks and proposing the problem, as well as the referee for the careful review and valuable comments which led to an improved presentation.





\medskip
\noindent Tja\v{s}a Vrhovnik

\noindent Departamento de Geometr\'{\i}a y Topolog\'{\i}a e Instituto de Matem\'aticas (IMAG), Universidad de Granada, Campus de Fuentenueva s/n, E--18071 Granada, Spain.

\noindent  e-mail: {\tt vrhovnik@ugr.es}

\end{document}